\documentclass[runningheads,a4paper]{llncs}

\usepackage{
amsmath,
amscd,
amssymb,
}
\usepackage{hyperref}
\usepackage{tikz}
\usetikzlibrary{positioning,arrows,calc}

\usepackage{graphicx}

\usepackage{url} 
\newcommand{\keywords}[1]{\par\addvspace\baselineskip
\noindent\keywordname\enspace\ignorespaces#1}

\newcommand{\Z}{\mathbb{Z}}

\newcommand{\N}{\mathbb{N}}

\newcommand{\B}{\mathcal{B}}
\newcommand{\X}{\mathcal{X}}
\newcommand{\V}{\mathcal{V}}
\newcommand{\E}{\mathcal{E}}
\newcommand{\Orb}{\mathcal{O}}
\newcommand{\Per}{\mathcal{P}}
\newcommand{\INF}{{}^\infty}
\newcommand{\ID}{\mathrm{id}}

\newcommand{\lcm}{\mathrm{lcm}}

\newcommand{\rank}{\mathrm{rank}}
\newcommand{\Pol}{\mathsf{P}}
\newcommand{\NP}{\mathsf{NP}}
\newcommand{\GI}{\mathsf{GI}}

\newcommand{\orb}{\circlearrowright}

\title{Complexity of Conjugacy, Factoring and Embedding for Countable Sofic Shifts of Rank 2}
\titlerunning{Conjugacy, Factoring and Embedding for Countable Sofic Shifts}

\author{
Ville Salo
\and
Ilkka T\"orm\"a
}

\institute{
		TUCS -- Turku Centre for Computer Science \\
		University of Turku, Finland \\
		\email{\{vosalo,iatorm\}@utu.fi}
}

\begin{document}
\maketitle

\begin{abstract}
In this article, we study countable sofic shifts of Cantor-Bendixson rank at most 2. We prove that their conjugacy problem is complete for $\GI$, the complexity class of graph isomorphism, and that the existence problems of block maps, factor maps and embeddings are $\NP$-complete.
\end{abstract}


\section{Introduction}

The computational complexity class $\GI$ is defined as the set of decision problems reducible to the graph isomorphism problem in polynomial time. The class is one of the strongest candidates for an $\NP$-intermediate class, that is, one that lies strictly between $\Pol$ and $\NP$. It contains a plethora of other isomorphism problems of finite objects, like multigraphs, hypergraphs, labeled or colored graphs, finite automata and context-free grammars, most of which are $\GI$-complete. The classical reference for the subject is \cite{ZeKoTy85}.

In the field of symbolic dynamics, which studies sets of infinite sequences of symbols as topological dynamical systems, there is a fundamental isomorphism problem whose decidability has been open for some decades: the conjugacy problem of shifts of finite type. A shift of finite type is a class of sequences defined by finitely many forbidden patterns that must never occur, and a conjugacy is a homeomorphism that commutes with the shift transformation, or in combinatorial terms, a bijection between shifts of finite type defined by a local rule. The most common version of this problem further restricts to \emph{mixing} shifts of finite type, whose dynamics is intuitively the most random and unpredictable. See \cite{Bo08} for a review of the problem (and many others).

In this article, we take a different approach, and study the conjugacy problem of \emph{countable sofic shifts}. Sofic shifts are a generalization of shifts of finite type, and the class of countable shifts of finite type can be seen as the polar opposite of the mixing class, as they are very well-structured and combinatorial. Their properties have previously been studied at least in \cite{SaTo12cn,SaTo13pub,PaSc10,BaDuJe08,BaJe13} (although usually somewhat indirectly).

A useful tool in the study of countable sofic shifts (and topological spaces in general) is the Cantor-Bendixson rank. Every countable sofic shift has such a rank, which is a number $n \in \N$, and we study the first few levels of this hierarchy in this article. Rank $1$ countable sofic shifts are the finite ones, and the conjugacy problem for them is very easy (Proposition~\ref{prop:Rank1}). Rank $2$ countable sofic shifts are the first non-trivial case, and our main result states that the conjugacy problem of rank $2$ countable SFTs and sofic shifts is $\GI$-complete (with respect to polynomial-time many-one reductions), when the shift spaces are given by right-resolving symbolic edge shifts. Using the same methods, we also prove that the existence of block maps, factor maps and embeddings between rank $2$ countable sofic shifts is $\NP$-complete. Of course, corresponding hardness results follow for general SFTs, since countable SFTs of rank $2$ are a (very small) subcase. However, we are not able to extract any corollaries for the usual case of mixing SFTs.

Note that it is of course $\GI$-complete to check whether two given edge shifts are isomorphic in the sense that the graphs defining them are isomorphic. This is not equivalent to conjugacy of the edge shifts, as shown in Example~\ref{ex:ConjugateEdgeShifts}. The graph representations we use are more canonical representations of the shift spaces, and very specific to the rank $2$ case.

\section{Definitions}

Let $A$ be a finite set, called the \emph{alphabet}, whose elements are called \emph{symbols}. We equip the set $A^\Z$ with the product topology and define the \emph{shift map} $\sigma : A^\Z \to A^\Z$ by $\sigma(x)_i = x_{i + 1}$ for all $x \in A^\Z$ and $i \in \Z$. The pair $(A^\Z, \sigma)$ is a dynamical system, called the \emph{full shift over $A$}. For a word $w \in A^*$ and $x \in A^\Z$, we say that $w$ \emph{occurs in} $x$, denoted $w \sqsubset x$, if there exists $i \in \Z$ with $x_{[i, i+|w|-1]} = w$.

A topologically closed and $\sigma$-invariant subset $X \subset A^\Z$ is called a \emph{shift space}. Alternatively, a shift space is defined by a set $F \subset A^*$ of \emph{forbidden patterns} as $X = \{ x \in A^\Z \;|\; \forall w \in F : w \not\sqsubset x \}$. If $F$ is finite, then $X$ is a \emph{shift of finite type}, or \emph{SFT} for short, and if $F$ is a regular language, then $X$ is a \emph{sofic shift}. For $n \in \N$, we denote $\B_n(X) = \{ w \in A^n \;|\; x \in X, w \sqsubset x \}$, and define the \emph{language} of $X$ as $\B(X) = \bigcup_{n \in \N} \B_n(X)$. We also denote by $\Per(X) = \{ x \in X \;|\; \exists p \in \N : \sigma^p(x) = x \}$ the set of $\sigma$-periodic points of $X$. We say $x$ and $y$ are left asymptotic if $x_i = y_i$ for all small enough $i$, and right asymptotic if this holds for all large enough $i$.

A shift space is uniquely determined by its language, so for a language $L$ such that $w \in L$ always implies $uwv \in L$ for some nonempty words $u, v \in A^*$ (often called an \emph{extendable} language), we denote $X = \B^{-1}(L)$, where $X$ is the unique shift space such that $\B(X) = \{w \;|\; \exists u, v \in A^*: uwv \in L\}$. For a configuration $x \in A^\Z$, we write $\Orb_\sigma(x) = \{\sigma^n(x) \;|\; n \in \Z\}$ for the \emph{$\sigma$-orbit of $x$}, and $\overline{X}$ for the topological closure of $X$, when $X$ is a subset of $A^\Z$.

A continuous function $f : X \to Y$ between shift spaces satisfying $\sigma|_Y \circ f = f \circ \sigma|_X$ is a \emph{block map}. Alternatively, a block map is defined by a \emph{local function} $\hat f : \B_{2r + 1}(X) \to \B_1(Y)$, where $r \in \N$ is called a \emph{radius} of $f$, as $f(x)_i = \hat f(x_{[i-r, i+r]})$. Block maps with radius $0$ are called \emph{symbol maps} and identified with their local functions. Sofic shifts are exactly the images of SFTs under block maps. The standard reference for shift spaces and block maps is \cite{LiMa95}.

If a block map $f : X \to Y$ is bijective, it is known that its inverse is also a block map, and then $f$ is called a \emph{conjugacy} between $X$ and $Y$. A surjective block map $f : X \to Y$ is called a \emph{factor map} from $X$ to $Y$, and such an injection is called an \emph{embedding}. The problem of deciding whether two SFTs are conjugate is known as the \emph{strong shift equivalence problem}, and its decidability has been open for several decades. See \cite{Bo08} for more information on this and other open problems in symbolic dynamics.

We include a full definition of Cantor-Bendixson rank for completeness, although we only need this concept for natural numbers in the case of sofic shifts. Let $X$ be a topological space. The \emph{Cantor-Bendixson derivative} of $X$ is the set $X' = \{x \in X \;|\; \overline{X \setminus \{x\}} = X\} \subset X$. In other words, $X'$ is exactly $X$ minus its isolated points. For every ordinal $\lambda$, we define the $\lambda$'th iterated derivative of $X$, denoted $X^{(\lambda)}$, as follows.
\begin{itemize}
\item If $\lambda = 0$, then $X^{(\lambda)} = X$.
\item If $\lambda = \beta + 1$, then $X^{(\lambda)} = (X^{(\beta)})'$.
\item If $\lambda$ is a limit ordinal, then $X^{(\lambda)} = \bigcap_{\beta < \lambda} X^{(\beta)}$.
\end{itemize}
The smallest ordinal $\lambda$ such that $X^{(\lambda)} = X^{(\lambda + 1)}$ is called the \emph{Cantor-Bendixson rank} of $X$, and denoted $\rank(X)$. If $X$ is a shift space, then it is countable if and only if $X^{(\rank(X))} = \emptyset$. In this case, the \emph{rank} of a point $x \in X$, denoted $\rank_X(x)$, is the least ordinal $\lambda$ such that $x \notin X^{(\lambda)}$. It is not hard to show that if $X$ is a shift space, then $X^{(\lambda)}$ is a shift space for all ordinals $\lambda$.

By a \emph{graph} we understand a tuple $G = (V, E)$, where $V = \V(G)$ is a finite set of \emph{vertices} and $E = \E(G)$ a set of \emph{edges}, which are two-element subsets of $V$ (so self-loops and multiple edges are not allowed). A \emph{labeled graph} is a triple $G = (V, E, \pi)$, where $\E(G) = E$ is now a set of tuples $(e, \ell)$ with $e$ an edge and $\ell \in L$ its \emph{label}, and $\pi = \pi_G : E \to L$ is the \emph{labeling function} given by $\pi(e, \ell) = \ell$. For a set $C$, a \emph{$C$-colored graph} has its vertices colored with elements of $C$ so that no adjacent vertices have the same color. Directed versions of all types of graphs have tuples as edges, instead of sets, and may have self-loops. Finally, a (directed) \emph{multigraph} is similar to a (directed) graph, but its edges form a multiset, so multiple edges between two vertices are allowed.

Homomorphisms of $C$-colored graphs must preserve the colors. For a fixed graph $H$ which is not bipartite, it is $\NP$-complete to decide whether there exists a homomorphism from a given graph $G$ to $H$, by a result of Hell and Ne\v{s}et\v{r}il \cite{HeNe90}. Likewise, for some classes of graphs $H$, it is $\NP$-complete whether an edge-surjective homomorphism (also known as a \emph{compaction}) exists from a given graph $G$ to $H$ \cite{Vi97}. Deciding the existence of edge-injective homomorphisms between two graphs is $\NP$-complete \cite{Bi05}, but not if one of them is fixed.

Every SFT is conjugate to an \emph{edge shift}, the SFT $X$ defined by a directed multigraph $G$ over the alphabet of edges $\E(G)$ as follows. A set of forbidden patterns for $X$ is given by the pairs of edges $e e'$ such that the target of $e$ differs from the source of $e'$, or in other words, $X$ is the set of edges of all bi-infinite walks in $G$. Similarly, every sofic shift is conjugate to a \emph{symbolic edge shift}, the image of an edge shift under a symbol map $\pi : \E(G) \to A$. Equivalently, a symbolic edge shift consists of the labels of all bi-infinite walks in a labeled directed graph. A symbolic edge shift is called \emph{right-resolving} if for each vertex $v \in G$, any distinct edges that start from $v$ have different labels.

The \emph{graph isomorphism problem} is the problem of deciding whether two graphs, encoded as lists of edges, are isomorphic. The set of decision problems polynomial-time reducible to the graph isomorphism problem is denoted $\GI$. It is known that $\Pol \subseteq \GI \subseteq \NP$. Examples of $\GI$-complete isomorphism problems include those of directed graphs, labeled graphs and $\{0, 1\}$-colored graphs. The classical reference for $\GI$ is \cite{ZeKoTy85}. In this paper, hardness and completeness with respect to a complexity class are taken with respect to standard polynomial time many-one reductions.

\section{Countable Sofic Shifts}

In this section, we give some background on countable sofic shifts and countable SFTs, and present their basic properties. As a conclusion, we obtain a combinatorial characterization of rank 2 sofic shifts in Corollary~\ref{cor:Rank2}. The point of this section is mainly to put our results into place in the theory of symbolic dynamics. Readers interested in the complexity theoretic result only can take the definition of a rank 2 sofic shift to be the second condition listed in Corollary~\ref{cor:Rank2}, and otherwise skip this section.

\begin{lemma}[Proposition 3.8 of \cite{SaTo13pub}]
\label{lem:FiniteRank}
 All sofic shifts have finite Cantor-Bendixson rank.
\end{lemma}

\begin{lemma}[Corollary of Lemma 4.8 of \cite{PaSc10}]
\label{lem:Repr}
A shift space $X \subset A^\Z$ is a countable sofic shift if and only if it can be presented as a finite union of shift spaces of the form
\[ \X(u_0, \ldots, u_m, v_1, \ldots, v_m) = \B^{-1}(u_0^* v_1 u_1^* \cdots u_{n-1}^* v_m u_m^*), \]
where $u_i \in A^+$ and $v_i \in A^*$, and each configuration $\INF u_i v_{i+1} u_{i+1} \INF$ is aperiodic.
\end{lemma}

Intuitively, the configurations of a countable sofic shift consist of long periodic areas, with `disturbances' of bounded length in between. The traditional application of symbolic dynamics is the encoding of information in a restricted medium, and from this viewpoint, countable sofic shifts are extremely restricted, as the asymptotic amount of information per coordinate in a configuration is zero.

\begin{definition}
Let $X \subset A^\Z$ be a countable sofic shift, and let $T$ be a finite set of tuples over $A^*$ such that $X = \bigcup_{t \in T} \X(t)$ and the conditions of Lemma~\ref{lem:Repr} hold. Then the set $T$ is called a \emph{combinatorial representation} of $X$.
\end{definition}

We remark that a nonempty countable sofic shift has infinitely many different combinatorial representations. Using the notation of Lemma~\ref{lem:Repr}, we write $n(u_0, \ldots, u_m, v_1, \ldots, v_m) = m$. The following lemma relates the rank of a countable sofic shift to its combinatorial representation.

\begin{lemma}
\label{lem:CombiRank}
Let $X \subset A^\Z$ be a nonempty countable sofic shift with the combinatorial representation $T$. Then
\begin{equation*}
\rank(X) = 1 + \max \{ n(t) \;|\; t \in T \}.
\end{equation*}
\end{lemma}

\begin{proof}
First, we need to show that if $X, Y \subset A^\Z$ are shift spaces with ranks $m, n \in \N$, respectively, then the (not necessarily disjoint) union $X \cup Y$ has Cantor-Bendixson rank $\max\{m, n\}$. This follows directly from the well-known property $(X \cup Y)' = X' \cup Y'$ of the derivative operator.

From this, we obtain by induction that the Cantor-Bendixson rank of a finite union of finite-rank shift spaces is just the maximal rank of the components. It is then enough to show that the Cantor-Bendixson rank of a shift space of the form $\X(u_0, \ldots, u_m, v_1, \ldots, v_m)$ is precisely $m + 1$, and we proceed by induction. It is clear that the rank is $1$ if $m = 0$, since the shift space is finite but nonempty. On the other hand, it is not hard to show that
\begin{align*} \X(u_0, \ldots, u_m, v_1, \ldots, v_m)' = \; &\X(u_1, \ldots, u_m, v_2, \ldots, v_m) \;\cup \\
&\X(u_0, \ldots, u_{m-1}, v_1, \ldots, v_{m-1}), \end{align*}
from which the claim follows. \qed
\end{proof}

We state some well-known characterizations of finite shift spaces, and then list some characterizations of the rank $2$ case.

\begin{corollary}
The following are equivalent for a nonempty shift space $X$:
\begin{itemize}
\item $X$ contains only periodic points,
\item $X$ is finite,
\item $X$ has rank $1$,
\item $X$ is a countable SFT (and/or sofic shift) of rank $1$,
\item $X$ is a finite union of shift spaces of the form $\X(u) = \B^{-1}(u^*)$.
\end{itemize}
\end{corollary}

\begin{corollary}
\label{cor:Rank2}
The following are equivalent for an infinite shift space $X$:
\begin{itemize}
\item $X$ is a countable sofic shift of rank $2$,
\item $X$ is a finite union of shift spaces of the form $\X(u_0, u_1, v) = \B^{-1}(u_0^* v u_1^*)$
 (where $\INF u_0 v u_1 \INF$ may or may not be periodic),
\item $X$ is a countable shift space of rank $2$,
\item every configuration in $X$ is either periodic or isolated.
\end{itemize}
\end{corollary}

\begin{proof}
The first and second conditions are equivalent by Lemma~\ref{lem:CombiRank}. The third and fourth are equivalent by the definition of rank and the previous corollary, and the first trivially implies the third.

We give a proof sketch for the fact that a countable shift space of rank $2$ satisfies the second condition, which concludes the proof. If $X \subset A^\Z$ is such a shift space, then $X'$ has rank $1$, and is thus an SFT by the previous corollary. By Lemma 2.6 in \cite{BaJe13}, $X \setminus X'$ then consists of finitely many orbits. Then, $X = \Orb_\sigma(x^1) \cup \cdots \Orb_\sigma(x^k) \cup Y$, where $Y \subset A^\Z$ is a finite shift space of periodic configurations, none of the configurations $x^i \in X$ are periodic, and the orbits $\Orb_\sigma(x_i)$ are pairwise disjoint. Since every $x^i$ is isolated in $X$ (as it is not in $X'$), there exists a word $w^{(i)} \in A^*$ which occurs in $x_i$ at exactly one position, and occurs in none of the $x_j$ for $j \neq i$. Now, note that the shift space
\[ X_i^R = \overline{\{\sigma^n(x^i) \;|\; n \in \N\}} \]
(the shift space generated by the right tail of $x^i$) is a subset of $X$ and does not contain any of the isolating patterns $w^{(j)}$ for $j \in [1, k]$. It is then a subset of $Y$, which implies that it is actually the orbit of a single periodic configuration, and $x^i$ then in fact has a periodic right tail.\footnote{There is a common period $p \in \N$ for the configurations in $Y$, and if this period breaks infinitely many times in the right tail of $x$, then $X_i^R$ is not contained in $Y$.} Similarly, $x^i$ has a periodic left tail, from which the claim follows. \qed
\end{proof}

\section{Structure Graphs}

Every SFT is conjugate to an edge shift, and if the graphs defining two edge shifts are isomorphic, then the SFTs are conjugate. It is clearly $\GI$-complete, in general, to check whether two SFTs are conjugate in this particular way. We show that even in the rank $2$ case, two SFTs can be conjugate even though the graphs defining their edge shifts are not isomorphic.\footnote{From Proposition~\ref{prop:Rank1}, one can extract that in the rank $1$ case, conjugacy of edge shifts \emph{is} equivalent to the graphs defining them being isomorphic.}

\begin{example}
\label{ex:ConjugateEdgeShifts}
Let $X$ be the edge shift of the directed graph
\begin{tikzpicture}[baseline=.3em]
\node[anchor=south] (a) at (0,0) {$a$};
\node[anchor=south] (b) at (.75,0) {$b$};
\path
  (a) edge[->,loop left] (a)
  (a) edge[->] ($ (a) + (.55,0) $)
  (b) edge[->,loop right] (b);
\end{tikzpicture},
and let $Y$ be that of
\begin{tikzpicture}[baseline=.3em]
\node[anchor=south] (a) at (0,0) {$a$};
\node[anchor=south] (b) at (.75,0) {$b$};
\node[anchor=south] (c) at (1.5,0) {$c$};
\path
  (a) edge[->,loop left] (a)
  (a) edge[->] ($ (a) + (.55,0) $)
  (c) edge[<-] ($ (c) - (.55,0) $)
  (c) edge[->,loop right] (c);
\end{tikzpicture}.
The graphs are not isomorphic, but the edge shifts are easily seen to be conjugate by the block map $f : X \to Y$ defined by
\[ f(x)_i = \left\{\begin{array}{ll}
(a, a), & \mbox{if } x_i = (a, a), \\
(a, b), & \mbox{if } x_i = (a, b), \\
(b, c), & \mbox{if } x_{i-1} = (a, b), \\
(c, c), & \mbox{if } x_{i-1} = (b, b). \\
\end{array}\right. \]
\end{example}

We now define the structure graph of a rank $2$ countable sofic shift, which is a certain labeled directed graph. Corollary~\ref{cor:Iso} shows that this graph is canonical up to renaming the vertices.

\begin{definition}
Let $X$ be a rank $2$ countable sofic shift. Define the labeled directed graph $G(X) = (V, E, \pi)$ as follows. First, $V = \Per(X)$ is the set of $\sigma$-periodic points of $X$. For all $x \in \Per(X)$, add an edge $x \to \sigma(x)$ into $E$ with the label $\orb$, called a \emph{rotation edge}. Then, for each pair of configurations $x, y \in \Per(X)$ such that the set $C_{(x, y)} = \{ z \in X \;|\; \exists n \in \N : z_{(-\infty,-n]} = x_{(-\infty,-n]}, z_{[n,\infty)} = y_{[n,\infty)} \}$ is nonempty, add an edge $e = (x, y)$ into $E$ with the label $\pi(e) = |\{ \Orb_\sigma(z) \;|\; z \in C_e \}|$, called a \emph{transition edge}. We call $G(X)$ the \emph{structure graph} of $X$.

A \emph{homomorphism} between two structure graphs $G(X)$ and $G(Y)$ is a graph homomorphism $\tau : G(X) \to G(Y)$ that satisfies
\[ \tau^{-1}(\pi_{G(Y)}^{-1}(\orb)) = \pi_{G(X)}^{-1}(\orb), \]
that is, $\tau$ respects the property of being a rotation edge. A bijective homomorphism is called an \emph{isomorphism}.
\end{definition}

Note that the structure graph is finite and transition edges have finite labels, as there are finitely many $\sigma$-orbits in a rank $2$ countable sofic shift. Also, the inverse function of an isomorphism of structure graphs is itself an isomorphism.

\begin{example}
\label{ex:Graph}
Let $X = \B^{-1}(0^* 1 0^* + 0^* (1 2)^* + 0^* (1 3)^* + (1 2)^* (1 + 2) (1 3)^*)$, a rank $2$ countable sofic shift. Then $\Per(X) = \{ \INF 0 \INF, \INF (1 2) \INF, \INF (2 1) \INF, \INF (1 3) \INF, \INF (3 1) \INF \}$, and the structure graph $G(X)$ is the one depicted in Figure~\ref{fig:ExampleGraph}.

\begin{figure}
\begin{center}
\begin{tikzpicture}

%nodes
\node (n0) at (-90:1) {$\INF 0 \INF$};
\node (n12) at (135:1.5) {$\INF (1 2) \INF$};
\node (n21) at (135:3) {$\INF (2 1) \INF$};
\node (n13) at (45:3) {$\INF (1 3) \INF$};
\node (n31) at (45:1.5) {$\INF (3 1) \INF$};

%\sigma-edges
\draw[->] (n0) edge [loop above] node {$\orb$} (n0);
\draw[->] (n12) edge [bend left] node [right] {$\orb$} (n21);
\draw[->] (n21) edge [bend left] node [above right=-0.1] {$\orb$} (n12);
\draw[->] (n13) edge [bend left] node [left] {$\orb$} (n31);
\draw[->] (n31) edge [bend left] node [above left=-0.1] {$\orb$} (n13);

%self-loops
\draw[->] (n0) edge [loop below] node [pos=.2,right] {$2$} (n0);
\draw[->] (n12) edge [loop, out=330, in=300, looseness=4] node [right] {$1$} (n12);
\draw[->] (n21) edge [loop, out=150, in= 120, looseness=4] node [left] {$1$} (n21);
\draw[->] (n13) edge [loop, out=60, in=30, looseness=4] node [right] {$1$} (n13);
\draw[->] (n31) edge [loop, out=240, in=210, looseness=4] node [left] {$1$} (n31);

%other edges
\draw [->] (n0) edge [bend left] node [left] {$1$} (n12);
\draw [->] (n0) edge [bend left=55] node [left] {$1$} (n21);
\draw [->] (n0) edge [bend right=55] node [right] {$1$} (n13);
\draw [->] (n0) edge [bend right] node [right] {$1$} (n31);
\draw [->] (n12) edge [bend left] node [above] {$2$} (n31);
\draw [->] (n21) edge [bend left] node [below] {$2$} (n13);

\end{tikzpicture}
\end{center}
\caption{The structure graph of the countable sofic shift $X$ of Example~\ref{ex:Graph}.}
\label{fig:ExampleGraph}
\end{figure}
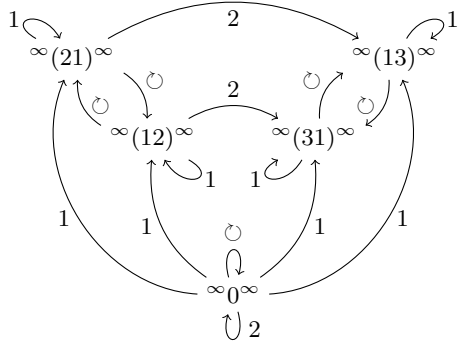
\end{example}

Next, we show that the structure graph is functorial: block maps between two shift spaces correspond to homomorphisms between their structure graphs.

\begin{proposition}
\label{prop:Functor}
For every block map $f : X \to Y$ between rank $2$ countable sofic shifts, there exists a homomorphism $G(f) : G(X) \to G(Y)$ between their structure graphs such that $G(\ID_X) = \ID_{G(X)}$ for all $X$, and $G(g \circ h) = G(g) \circ G(h)$ for all $g : Y \to Z$ and $h : X \to Y$.
\end{proposition}

\begin{proof}
Let $f : X \to Y$ be as stated, and define $G(f) : G(X) \to G(Y)$ as follows. For each $x \in \Per(X)$, let $G(f)(x) = f(x) \in \Per(Y) = \V(G(Y))$. Then for any rotation edge $e : x \to \sigma(x)$ in $\E(G(X))$, there exists a rotation edge $e' : f(x) \to \sigma(f(x))$ in $\E(G(Y))$, so let $G(f)(e) = e'$. Finally, for any transition edge $e : x \to y$ with label $k > 0$, the set $C_{(x, y)}$ is nonempty, which implies that $C_{(f(x), f(y))}$ is also nonempty. Thus $e' = (f(x), f(y))$ is a transition edge of $G(Y)$ with some label $\ell > 0$, and we again let $G(f)(e) = e'$.

It is easy to see that $G(f)$ is a homomorphism between the structure graphs, and that $G(\ID_X) = \ID_{G(X)}$ holds. Proving the equation $G(g \circ h) = G(g) \circ G(h)$ is simply a matter of expanding the definitions.
\qed
\end{proof}

The operation $G$ that sends block maps to structure graph homomorphisms preserves injectivity and surjectivity in the following sense.

\begin{lemma}
\label{lem:InjSurj}
A homomorphism $\tau : G(X) \to G(Y)$ of structure graphs
\begin{enumerate}
\item always has a $G$-preimage,
\item has an injective $G$-preimage if and only if it is edge-injective, and every transition edge $e$ of $G(X)$ satisfies $\pi(e) \leq \pi(\tau(e))$,
\item has a surjective $G$-preimage if and only if it is edge-surjective, and every transition edge $e$ of $G(Y)$ satisfies $\sum_{\tau(e') = e} \pi(e') \geq \pi(e)$,
\item has a bijective $G$-preimage if and only if it is edge-bijective and preserves the labels of transition edges.
\end{enumerate}
\end{lemma}

\begin{proof}
First, we prove that if $f : X \to Y$ is injective, then condition 2 holds for $G(f)$. The edge-injectivity of $G(f)$ follows immediately. For each transition edge $e = (v, w)$ of $G(X)$ with label $k > 0$, the function from $\{ \Orb_\sigma(y) \;|\; y \in C_e \}$ to $\{ \Orb_\sigma(z) \;|\; z \in C_{G(f)(e)} \}$ defined by $\Orb_\sigma(y) \mapsto \Orb_\sigma(f(y))$ is injective, since $f$ is. Thus we have $k \leq \pi(G(f)(e))$, and condition 3 holds.

Suppose next that $f$ is surjective, so that $G(f)$ is clearly surjective on vertices and rotation edges. For each transition edge $e = (v, w)$ of $G(Y)$ with label $k > 0$ and each configuration $y \in C_{(v, w)}$, there exists some periodic configurations $v' \in f^{-1}(v)$ and $w' \in f^{-1}(w)$, and a configuration $z \in f^{-1}(y) \cap C_{(v', w')}$. Then $(v', w')$ is an edge of $G(X)$ that $G(f)$ maps to $e$, so it is surjective on the transition edges as well. Finally, we have
\begin{align*}
k & = |\{ \Orb_\sigma(y) \;|\; y \in C_e \}|
   \leq |\{ \Orb_\sigma(z) \;|\; z \in X, f(z) \in C_e \}| \\
  & = \sum_{\substack{v', w' \in \Per(X) \\ G(f)(v', w') = e}} |\{ \Orb_\sigma(z) \;|\; z \in C_{(v', w')} \}|
   = \sum_{G(f)(e') = e} \pi(e').
\end{align*}
If $f$ is bijective, this and the previous case together show that condition 4 holds.

Finally, we construct a $G$-preimage $f : X \to Y$ for $\tau$ with the desired properties. We must of course have $f(x) = \tau(x)$ for every periodic configuration $x \in \Per(X)$, which is well-defined since $\tau(\sigma(x)) = \sigma(\tau(x))$. Let then $x \in X$ be aperiodic, and let $y, z \in \Per(X)$ be such that $x \in C_{(y, z)}$. Then $G(X)$ has a transition edge $e = (y, z)$ with some label $k > 0$. Let $e' = \tau(e) = (y', z')$. Now $e'$ has some label $\ell > 0$, so that $C_{e'}$ is also nonempty. We choose some $x' \in C_{e'}$ and define $f(x) = x'$, and extend $f$ to $\Orb_\sigma(x)$ by defining $f(\sigma^n(x)) = \sigma^n(x')$ for all $n \in \Z$. For this, note that $x$ is isolated in $X$. If condition 2 holds, there are enough orbits in $\bigcup_{e \in \tau^{-1}(e')} C_e$ to guarantee that every $x' \in C_{e'}$ can be given an $f$-preimage, and if condition 3 holds, then there are enough orbits in $C_{e'}$ to guarantee that every $x \in C_{(y, z)}$ can be given a different $f$-image.

The definition of $f$ is now complete, and it is easy to see that it is continuous and shift-invariant, thus a block map. Moreover, it follows immediately from the definition of $f$ that $G(f) = \tau$. \qed
\end{proof}

As a corollary of the above, we obtain the following.

\begin{corollary}
\label{cor:Iso}
Let $X$ and $Y$ be countable rank $2$ sofic shifts. Then $X$ and $Y$ are conjugate if and only if $G(X)$ is obtained from $G(Y)$ by renaming its vertices.
\end{corollary}

\section{Complexity Classes of Conjugacy, Factoring, Embedding and Block Map Existence}

In this section, we present our results on the computational complexity of different decision problems related to countable rank $1$ and rank $2$ sofic shifts. If one is only interested in decidability, then it is irrelevant what kind of encodings we use for shift spaces $X \subset A^\Z$, since there are computable transformations between all reasonable encodings. However, for finding out the precise complexity class, there are some subtleties in how the algorithm receives the shift space as input. There are several possibilities:
\begin{enumerate}
\item Every sofic shift can be encoded by a finite list $F \subset B^*$ of forbidden words that define an SFT over the alphabet $B$, and a symbol map $f : B \to A$. For SFTs, we can take $B = A$ and $f = \ID_A$.
\item Every sofic shift is the symbolic edge shift defined by a (possibly right-resolving) labeled graph. For SFTs, up to conjugacy, we can take the symbol map to be the identity map, and obtain edge shifts given by adjacency matrices. This is the standard encoding of SFTs in the conjugacy problem.
\item Countable sofic shifts can be encoded by combinatorial representations.
\item Countable sofic shifts of rank at most $2$ can be encoded by structure graphs.
\end{enumerate}
We show that, up to polynomial-time reductions, encodings 2 and 4 are equivalent in the rank $2$ case, if we assume right-resolvingness, so it makes no difference which one we choose.

\begin{lemma}
\label{lem:Representations}
For countable sofic shifts of rank at most $2$, the representations by right-resolving symbolic edge shifts and structure graphs are equivalent up to polynomial-time reductions.
\end{lemma}

\begin{proof}
First, let $G$ be a right-resolving labeled graph of size $n$ encoding a countable rank $2$ sofic shift $X \subset A^\Z$. We construct the structure graph of $X$, and for this, we may assume that $G$ is the minimal right-resolving representation of $X$ (which can be computed in polynomial time from a given right-resolving labeled graph, see Section~4 of \cite{LiMa95} for details). In particular, since $X$ is countable, the cycles of $G$ are disjoint, and we can enumerate them as $C_0, \ldots, C_{c-1}$, and if $(q_i^0, \ldots, q_i^{m_i-1})$ are the vertices of the cycle $C_i$, then $\sum_{i=0}^{c-1} m_i \leq n$. Also, if $u_i \in A^{m_i}$ is the label of the cycle $C_i$, then $\INF u_i^\infty$ has least period $m_i$, for otherwise we could replace $C_i$ by a shorter cycle and obtain a smaller presentation for $X$.

Call a path in $G$ \emph{transitional} if it contains no edges of any cycle. Then the length of a transitional path is at most $n$. Now, let $M \in \N^{\V(G) \times \V(G)}$ be the matrix defined by $M(q_i^r, q_i^{r+1}) = 0$ for all $i \in [0,c-1]$ and $r \in [0,m_i-1]$, and $M(v, w) = |\{ e : v \to w \;|\; e \in \E(G) \}|$ for all other vertices $v, w \in \V(G)$. Note that since $G$ is right-resolving, the edges in the above set have distinct labels. The number of transitional paths of a given length $\ell \in \N$ between two vertices $v, w \in \V(G)$ is then exactly $M^\ell(v, w)$, and their labels are also distinct.

For each periodic configuration $\INF u \INF \in X$, there are at most two cycles in $G$ with the label $u$, one with outgoing and one with incoming transitional paths, for otherwise we could replace two such cycles with a single one and obtain a smaller right-resolving representation. Now, the vertices of the structure graph $G(X)$ are the periodic configurations $\sigma^r(\INF u_i^\infty)$ for $i \in [0, c-1]$ and $r \in [0, m_i - 1]$, and its transitional edges are exactly $e : x \to y$ with labels
\[ \pi(e) = \sum_{\substack{x = \sigma^r(\INF u_i^\infty) \\ y = \sigma^s(\INF u_j^\infty)}} \sum_{\ell = 1}^n \sum_{k = 0}^{\lcm(m_i, m_j)-1} |M^\ell(q_i^{r+k}, q_j^{s+k+\ell})|, \]
plus $1$ in the case $x = y$ to account for the periodic point itself. Thus $G(X)$ can be computed from $G$ in polynomial time.

Next, we take the structure graph $G(X)$ for $X$, and construct a right-resolving labeled graph $G$ whose shift space $Y$ is conjugate to $X$. From $G(X)$ we can easily extract the periodic orbits of $X$, which we denote by $\Orb_\sigma(\INF u_i^\infty)$ for $i \in [0, c-1]$, where $u_i \in A^+$. For all $i \in [0, c-1]$ and $r \in [0, |u_i|-1]$, we add to $G$ two vertices $p_i^r$ and $q_i^r$, and two edges $e : q_i^r \to q_i^{r+1}$ and $e' : p_i^r \to p_i^{r+1}$ with the same labels $a_i^r$. The labels of the cycles of $G$, and thus the periodic points of $Y$, are thus $a_i = a_i^0 \cdots a_i^{|u_i|-1}$ for $i \in [0, c-1]$.

Let then $e : \sigma^r(\INF u_i^\infty) \to \sigma^s(\INF u_j^\infty)$ be a transitional edge in the structure graph $G(X)$, and let $\pi(e) = \sum_{\ell = 0}^{d-1} 2^{k_\ell}$ be the binary representation of its label, where $k_0 < \cdots < k_{d-1}$. For each $\ell \in [0, d-1]$, we add to $G$ the subgraph $G_e = q_i^r \to v_e^0 \rightrightarrows v_e^1 \rightrightarrows \cdots \rightrightarrows v_e^{k_\ell} \to w_e^0 \to \cdots \to w_e^p \to p_j^s$ whose length divides $|u_i|$ and $|u_j|$. Apart from the vertices $q_i^r$ and $p_j^s$, the subgraphs $G_e$ for different transitional edges $e$ are disjoint, and their edges have distinct labels. Then $G$ is a right-resolving labeled graph with exactly $\pi(e)$ transitional paths from $q_i^r$ to $p_j^s$ of length dividing $|u_i|$ and $|u_j|$. Every configuration of $Y$ which is left asymptotic to $\sigma^r(\INF a_i^\infty)$ and right asymptotic to $\sigma^s(\INF a_i^\infty)$ contains the label of one of the paths in $G_e$. This implies that $G(Y)$ is obtained from $G(X)$ by renaming each vertex $\sigma^r(\INF u_i^\infty)$ to $\sigma^r(\INF a_i^\infty)$, and then $X$ and $Y$ are conjugate by Corollary~\ref{cor:Iso}. It is clear that the construction of $G$ can be done in polynomial time. \qed
\end{proof}

We can also show a similar result for forbidden words and general symbolic edge shifts, although we omit the proof.

\begin{lemma}
\label{lem:Representations2}
For countable sofic shifts of rank at most $2$, the representations by  symbolic edge shifts and forbidden words are equivalent up to polynomial-time reductions.
\end{lemma}

It is known \cite{KaSwMa95} that computing the number of words of a given length accepted by a given nondeterministic finite automaton is complete in a complexity class known as $\mathsf{\#P}$, which contains $\NP$ and is believed to be much larger than it. Thus, under reasonable complexity assumptions, there is no polynomial-time algorithm for computing the structure graph of a given symbolic edge shift, if the input need not be right-resolving. We do not know the exact complexity of the conjugacy problem, if the inputs are given in this form. The hardness results are still valid, and the problems are all decidable in this case as well.

The next example shows that the combinatorial representation is not equivalent to the other three, since there is an exponential blowup.

\begin{example}
Any combinatorial representation of the rank $2$ countable SFT defined by the right-resolving labeled graph
\[
\begin{tikzpicture}[xscale=.75]
\node[right] (q0) at (0,0) {$q_0$};
\node[right] (q1) at (1,0) {$q_1$};
\node[right] (q2) at (2,0) {$q_2$};
\node[right] (d) at (3.2,0) {$\cdots$};
\node[right] (q3) at (4.4,0) {$q_k$};
\foreach \x in {0,1,2,3.4}{
  \node[right] () at (\x+.5,0) {$\rightrightarrows$};
  \node[right] () at (\x+.625,.25) {\scriptsize 1};
  \node[right] () at (\x+.625,-.25) {\scriptsize 2};
}
\path[->]
  (q0) edge[loop left] node[above=0.06]{\scriptsize 0} (q0)
  (q3) edge[loop right] node[above=0.06]{\scriptsize 3} (q3);
\end{tikzpicture}
\]
clearly contains at least $2^k$ terms, as each term only represents one $\sigma$-orbit.
\end{example}

In what follows, we assume that in the decision problems, all countable SFTs and sofic shifts are encoded by their structure graphs, and we will do so without explicit mention.

We first solve the case of rank 1. As one might imagine, there are fast and simple algorithms in this case.

\begin{proposition}
\label{prop:Rank1}
Conjugacy, and existence of block maps, factor maps, and embeddings between countable sofic shifts of Cantor-Bendixson rank~$1$ is in~$\Pol$.
\end{proposition}

\begin{proof}
Every such shift space $X$ is a finite union of periodic orbits of some least periods $p_1, \ldots, p_\ell \in \N$, which can be computed from the structure graph in polynomial time. Let $Y$ be another one with least periods $q_1, \ldots, q_m \in \N$. Clearly, $X$ and $Y$ are conjugate if and only if $(p_1, \ldots, p_\ell) = (q_1, \ldots, q_m)$, if the periods are given in ascending order. A block map from $X$ to $Y$ exists if and only if, for every $i \in [1,\ell]$, there exists $j \in [1,m]$ with $q_j | p_i$, which is equivalent to the condition that the orbit of $X$ with period $p_i$ can be mapped onto the orbit of $Y$ with period $q_j$. An embedding from $X$ to $Y$ must map every orbit of $X$ to an orbit of $Y$ of the same period, so one exists if and only if there exists an injection $\alpha : [1,\ell] \to [1,m]$ with $p_i = q_{\alpha(i)}$ for all $i \in [1,\ell]$. These checks are easy to do in polynomial time.

The interesting case is factoring. For this, construct a bipartite graph $G$ with
\[ \V(G) = \{L\} \times [1, \ell] \cup \{R\} \times [1, m], \]
and $((L, i), (R, j)) \in \E(G)$ if and only if $q_j | p_i$. It is easy to see that there exists a factor map from $X$ to $Y$ if and only if there exists a block map from $X$ to $Y$, and $G$ has a matching of size $m$ (that is, we can find separate preimages for all the orbits of $Y$). Computing a matching of maximal size -- and thus the maximal size itself -- is well-known to be in $\Pol$, see for example Section~5.2 in \cite{Gi85}.
\qed
\end{proof}

Now, we give our main result: the complexity of conjugacy of rank $2$ countable SFTs (and sofic shifts).

\begin{theorem}
Conjugacy of countable sofic shifts or SFTs of Cantor-Bendixson rank at most $2$ is $\GI$-complete.
\end{theorem}

\begin{proof}
Corollary~\ref{cor:Iso} states that conjugacy is equivalent to the equivalence problem of structure graphs under vertex renaming. This problem, on the other hand, is easily reducible to the isomorphism problem of directed graphs (associate to each label $k$ a distinct small number $m_k \geq 3$, and replace every edge $e : x \to y$ with label $k$ by $m_k$ parallel paths of length $m_k$), which is $\GI$-complete \cite{ZeKoTy85}. Thus the conjugacy problem is in $\GI$.

To prove completeness, we reduce the color-preserving isomorphism problem of $\{0,1\}$-colored graphs to the conjugacy problem of SFTs; the claim then follows, as the former is $\GI$-complete \cite{ZeKoTy85}. Let thus $G = (V, E, \pi)$ be a $\{0,1\}$-colored graph with the coloring $\pi : V \to \{0,1\}$. We may assume that $G$ contains no isolated vertices. Define a rank $2$ countable SFT by
\[ X_G = \bigcup_{\substack{\{u,v\} \in \E(G) \\ \pi(u) = 0}} \B^{-1}(u^* v^*). \]
After renaming the vertices, the structure graph of $X_G$ is exactly $G$, except that each vertex $v$ has gained two self-loops labeled $\orb$ and $1$, and each edge $\{u,v\}$ where $u$ is colored with $0$ has gained the label $1$. Thus, two $\{0,1\}$-colored graphs $G$ and $H$ are isomorphic by a color-preserving isomorphism if and only if the structure graphs of $X_G$ and $X_H$ are equivalent up to renaming the vertices. By Corollary~\ref{cor:Iso}, this is equivalent to the conjugacy of $X_G$ and $X_H$.
\qed
\end{proof}

With the same ideas, we obtain many $\NP$-complete problems, at least for countable sofic shifts.

\begin{theorem}
Existence of embeddings between countable sofic shifts of Cantor-Bendixson rank at most $2$ is $\NP$-complete. Also, there exist countable rank $2$ sofic shifts $X$ and $Y$ such that for a given countable rank $2$ sofic shift $Z$, existence of block maps from $Z$ to $X$, and of factor maps from $Z$ to $Y$, are $\NP$-complete problems.
\end{theorem}

\begin{proof}
Proposition~\ref{prop:Functor} and Lemma~\ref{lem:InjSurj} imply that all three problems are in $\NP$, since the conditions given in Lemma~\ref{lem:InjSurj} are easy to check in polynomial time for a given structure graph homomorphism.

We prove the completeness of all three problems using the same construction. For all graphs $G$, we define a countable rank $2$ sofic shift
\[ X_G = \bigcup_{\{u,v\} \in \E(G)} \B^{-1}((\# u)^* (\# v)^*) \cup \B^{-1}((\# u)^* (v \#)^*), \]
where $\# \notin \V(G)$ is a new symbol.

For two graphs $G$ and $H$, we show a correspondence between graph homomorphisms $\phi : G \to H$ and block maps $f : X_G \to X_H$. First, for each homomorphism $\phi$, we define a block map $f_\phi$ as the symbol map $f_\phi(\#) = \#$ and $f_\phi(v) = \phi(v)$ for all $v \in \V(G)$, and by the definition of $X_G$ and $X_H$, it is a well-defined block map from $X_G$ to $X_H$. Second, for a block map $f$, we define a homomorphism $\phi_f$ by $\phi_f(v) = w$ if and only if $f(\INF (\# v) \INF) \in \Orb_\sigma(\INF (\# w) \INF)$. Then $\{u, v\} \in \E(G)$ implies that $\INF (\# u) (\# v) \INF, \INF (\# u) (v \#) \INF \in X_G$, and since $X_H$ contains no points of period $1$, the $f$-image of at least one the configurations is aperiodic. But this implies $\{\phi_f(u), \phi_f(v)\} \in \E(H)$ by the definition of $X_H$, so that $\phi_f$ is indeed a graph homomorphism.

It is easy to see that surjectivity or injectivity of $\phi$ (on both vertices and edges) implies the same property for $f_\phi$, and analogously for $f$ and $\phi_f$. The claim then directly follows from the corresponding $\NP$-completeness results for graph homomorphisms, compactions and edge-injective homomorphisms found in \cite{HeNe90,Vi97,Bi05}, respectively.
\qed
\end{proof}

\section{Further Discussion}

In this article, we have studied the conjugacy problem of countable SFTs and sofic shifts, and have shown that the special case of Cantor-Bendixson rank $2$ is decidable. The classical formulation of the conjugacy problem of SFTs only considers mixing SFTs of positive entropy, which are uncountable, and conceptually very far from countable SFTs. Even though our results do not directly advance the study of this notoriously difficult problem, they show that related decision problems can be computable, and even have a relatively low computational complexity.

A natural continuation of this research would be to extend the results to countable sofic shifts of higher ranks, possibly for all countable sofic shifts. We suspect that for rank $3$ countable sofic shifts, the problem is no longer in $\GI$, as distances between two disturbances can encode infinitely many essentially different configurations. However, we also believe it to be decidable, possibly even in $\NP$, as any conjugacy has a finite radius and must thus consider distant disturbances separately. The problem is then essentially combinatorial, and finding a suitable representation for the shift spaces, similar to the structure graph, might be the key to determining its complexity class.

Of course, a lot of tools have been developed for tackling the mixing case, and it could be that these tools easily decide conjugacy in the countable case. For example, it is well-known that a weaker type of conjugacy called shift equivalence is decidable in high generality, so it would be enough to show that in the case of countable SFTs, this is equivalent to conjugacy (although we have not been able to show this). Thus, we explicitly state our interest:

\begin{question}
Is the conjugacy of countable SFTs decidable?
\end{question}

In the usual case of mixing (and uncountable) SFTs and sofics, the decidability of conjugacy has not yet been solved. However, there might be an easy way to show that the conjugacy problem is, say, $\NP$-hard. We are not aware of such investigations in the literature. Such a view might be helpful in finding ways to encode computation in instances of the conjugacy problem. Such a way would presumably need to be found in order to show that the problem is undecidable, but might also be useful (or at least interesting) if it turns out to be decidable.

Finally, we note that in the case of multidimensional SFTs, the conjugacy problem is undecidable. In fact, it was even shown in \cite{JeVa12} that for all two-dimensional SFTs $X$, it is undecidable whether a given SFT $Y$ is conjugate to it (and they also determine the complexity of finding factor maps).

\bibliographystyle{plain}
\bibliography{../../../bib/bib}{}

\end{document}